\documentclass[leqno]{article}
\usepackage{geometry}
\usepackage{graphicx}
\usepackage[T2A]{fontenc}
\usepackage[utf8]{inputenc}
\usepackage[english]{babel}
\usepackage{hyphenat}
\usepackage{mathtools}
\usepackage{amsfonts,amssymb,mathrsfs,amscd,amsmath,amsthm}
\usepackage{indentfirst}
\usepackage{verbatim}
\usepackage{url}
\usepackage{hyperref}
\usepackage{stmaryrd}
\usepackage{alltt}
\graphicspath{{pictures/}}
\DeclareGraphicsExtensions{.pdf,.png,.jpg}
\usepackage{caption}
\usepackage{wrapfig}
\usepackage{tikz}

\theoremstyle{plain}

\newtheorem{theorem}{Theorem}[section]

\newtheorem{lem}[theorem]{Lemma}

\theoremstyle{definition}

\newtheorem{defin}[theorem]{Definition}

\title{A problem of intersection of balls in normed space}
\author{Daniil Iliukhin}
\date{}

\begin{document}

\maketitle

\begin{abstract}
This paper investigates the topological properties of intersections of balls in finite-dimensional normed spaces---a problem that naturally arises when constructing covers for estimating the Gromov--Hausdorff distance. We study the topology of a set obtained by removing a large closed ball from a finite intersection of small open balls. It is proved that in an arbitrary normed plane, such a set is always contractible, provided that it is non-empty.
\end{abstract}

\section{Introduction}

In modern metric geometry, one of the important problems is the study of the Gromov--Hausdorff distance between various metric spaces, in particular, between normed spaces and their subsets (\cite{Tuzhilin}, \cite{GH}). Computing the exact values of this distance is a difficult computational problem; therefore, there is an ongoing search for non-trivial estimates relating the Gromov--Hausdorff distance to simpler metric characteristics, such as the Hausdorff distance (\cite{adams}).

An effective method for obtaining estimates for the Gromov--Hausdorff distance is the application of algebraic topology tools. In particular, simplicial complexes and the classical Nerve theorem are used. The Nerve theorem states that if a family of open sets forms a so-called "good cover"---that is, all non-empty finite intersections of the cover elements are contractible sets---then the nerve of this cover is homotopy equivalent to the union of these sets (\cite{hatcher}).

In finite-dimensional normed spaces, open balls are convex sets. Since the intersection of any family of convex sets is also convex, it is automatically contractible (or empty). This makes collections of open balls a natural and convenient tool for constructing good covers. However, when proving finer estimates for the Gromov--Hausdorff distance, a technical need arises to expand the covers by adding external regions of the space that are homeomorphic to the complements of balls of a sufficiently large radius.

The complement of a closed ball is not a convex set. In this regard, a non-trivial topological problem arises: is contractibility preserved when intersecting a finite number of open balls with the complement of a large closed ball? In the general case, the intersection of a convex set with the exterior of a ball can consist of several connected components or have a complicated topology. Nevertheless, it turns out that if the radii of the initial intersected balls are strictly less than the radius of the subtracted closed ball, the situation changes fundamentally.

The purpose of this paper is to investigate the topological properties of the difference of sets consisting of balls in a finite-dimensional normed space. The main result of the paper is the proof that if, in a normed plane, the intersection of several open balls of a small radius is not entirely contained in a closed ball of a larger radius, then their difference (that is, the part of the intersection of the small balls lying strictly outside the large ball) always forms a contractible set.

This result has an independent geometric nature, describing the structure of the intersection of convex bodies of a special type.

\section{Normed spaces and contractibility}

\begin{defin}[\cite{burago}]
A \textbf{normed space} is a pair $(X, \|\cdot\|)$ consisting of a vector space $X$ over the field of real or complex numbers and a mapping $\|\cdot\| \colon X \to \mathbb{R}$ such that the following properties hold for any $x, y \in X$ and scalar $\lambda$:
\begin{itemize}
    \item $\|x\| \geqslant 0$, $\|x\| = 0 \Rightarrow x = 0$ (positive definiteness);
    \item $\|\lambda x\| = |\lambda| \cdot \|x\|$ (absolute homogeneity);
    \item $\|x + y\| \leqslant \|x\| + \|y\|$ (triangle inequality).
\end{itemize}
\end{defin}

\begin{defin}[\cite{burago}]
A \textbf{homotopy} is a family of continuous mappings $F_t \colon X \to Y$, $t \in [0, 1]$, that depend continuously on the parameter $t$. More formally, it is a continuous mapping $F \colon X \times [0, 1] \to Y$. 

Mappings $f, g \colon X \to Y$ are called \textbf{homotopic} ($f \sim g$) if there exists a homotopy $F_t$ such that $F_0 = f$ and $F_1 = g$.

A \textbf{homotopy equivalence} between topological spaces $X$ and $Y$ is a pair of continuous mappings $f \colon X \to Y$ and $g \colon Y \to X$ such that the compositions $f \circ g$ and $g \circ f$ are homotopic to the identity mappings $\operatorname{id}_Y$ and $\operatorname{id}_X$, respectively ($f \circ g \sim \operatorname{id}_Y$ and $g \circ f \sim \operatorname{id}_X$). In this case, the spaces $X$ and $Y$ are said to be homotopy equivalent (or to have the same homotopy type).

A topological space (or set) $X$ is called \textbf{contractible} if its identity mapping $\operatorname{id}_X$ is homotopic to a mapping to some point $x_0 \in X$. In other words, the space $X$ is homotopy equivalent to a point.
\end{defin}

\section{A problem of intersection of balls in normed space}

\begin{lem}
In a normed space of finite dimension $n$, let $B=\{x\;:\;\|x-x_B\|\leq r_B\}$ be a closed ball and $A=\{x\;:\;\|x-x_A\|< r_A\}$ be an open ball, with $r_A<r_B$. If the set $A \setminus B$ is non-empty, then it is contractible.
\end{lem}

\noindent\textit{Proof.} 

Since the balls $A$ and $B$ are defined in the same norm and have different radii, there exists a unique homothety $H(x)=c+\lambda (x-c)$ that maps the closure of ball $A$ to ball $B$: $H(\overline{A})=B$.

The coefficient of the homothety is equal to the ratio of their radii: $\lambda = \frac{r_B}{r_A}>1$. The center of the homothety is $c=\frac{\lambda x_A-x_B}{\lambda-1}$.

Let us show that $c\notin \overline{A}$. Assume the contrary. Since $\overline{A}$ is a convex set and $\lambda>1$, the image $H(\overline{A})$ under the homothety centered at point $c$ will strictly contain the initial set: $\overline{A}\subset H(\overline{A})$. But $H(\overline{A})=B$, hence $\overline{A}\subset B$. This contradicts the condition that the difference $A\setminus B$ is non-empty.

Let us introduce a Cartesian coordinate system in the space $\mathbb{R}^n$ with the origin at the point $c$. We direct the axis $Ox_1$ from the point $c$ towards the center of the small ball $x_A$. The remaining axes $Ox_2, \dots, Ox_n$ are chosen to be orthogonal to $Ox_1$ in an arbitrary manner. Any direction of a ray emanating from $c$ is determined by a unit vector $v \in S^{n-1}$.

Let $V \subset S^{n-1}$ be the set of "effective directions", that is, vectors $v$ for which the ray $R_v(t) = c + tv$ ($t > 0$) intersects the open ball $A$. The closure of this set, $\overline{V}$, corresponds to the rays intersecting the closure $\overline{A}$.

Since $c$ lies strictly outside the convex compact set $\overline{A}$, there exists a hyperplane strictly separating $c$ and $\overline{A}$. In the chosen coordinate system, the rays intersecting $\overline{A}$ cannot deviate from the $Ox_1$ axis by an angle of $\pi/2$ or more. Therefore, for all $v \in \overline{V}$, the first coordinate is strictly positive: $v_1 > 0$.

\begin{wrapfigure}{r}{0.55\textwidth}
\centering
\vspace{-10pt} 
\begin{tikzpicture}[scale=0.8]
    \coordinate (C) at (0,0);
    \coordinate (XA) at (2,0);
    \coordinate (XB) at (4,0);

    \draw[thick, blue] (XA) circle (1cm);
    \draw[thick, orange] (XB) circle (2cm);

    \draw[thick, dashed] (C) -- (7,0); 
    \draw[thick, dashed] (C) -- (30:7.5); 

    \fill (C) circle (1.5pt) node[below=2pt] {$c$};
    \fill (XA) circle (1.5pt) node[below=2pt] {$x_A$};
    \fill (XB) circle (1.5pt) node[below=2pt] {$x_B$};
    \node[blue] at (2.3, -0.6) {$A$};
    \node[orange] at (4.5, -1.2) {$B$};
\end{tikzpicture}
\caption{}
\label{fig:homothety}
\end{wrapfigure}

We project the unit hemisphere $\{v \in S^{n-1} : v_1 > 0\}$ onto the hyperplane $x_1 = 1$ from the origin. To each vector $v$ there uniquely corresponds a vector $u = (u_2, \dots, u_n) \in \mathbb{R}^{n-1}$ such that $v = \frac{(1, u)}{\|(1, u)\|}$. This correspondence is a homeomorphism.

The set of rays intersecting the convex compact set $\overline{A}$ forms a solid convex cone. The intersection of this cone with the hyperplane $x_1 = 1$ is a convex, closed, and bounded set $D$ in $\mathbb{R}^{n-1}$. Consequently, the parameter domain $D$, corresponding to $\overline{V}$, is homeomorphic to a closed disk in $\mathbb{R}^{n-1}$, and its interior $\operatorname{int} D$, corresponding to $V$ (the rays intersecting $A$), is homeomorphic to an open disk.

In what follows, we parameterize the directions by a vector $u \in \operatorname{int} D$. The corresponding ray will be denoted by $R_u(t)$.

For any parameter $u \in \operatorname{int} D$, the intersection of the ray $R_u(t)$ with the open set $A$ is an interval $(c + t_1(u) v(u), c + t_2(u) v(u))$. The functions $t_1(u)$ and $t_2(u)$ depend continuously on the parameter $u$ on the open disk $\operatorname{int} D$.

Since the ball $B$ is the image of the ball $A$ under the homothety $H$, the intersection of the exact same ray with $B$ is the segment $[c + \lambda t_1(u) v(u), c + \lambda t_2(u) v(u)]$.

Let us find the points of the ray that belong to the set $A \setminus B$. Since $\lambda > 1$ and $t_1(u) > 0$, the front boundary of the ball $B$ is always strictly further from the center $c$ than the front boundary of the ball $A$. Consequently, the intersection of $A \setminus B$ with a given ray is the open interval:
$$I_u = (c + t_1(u) v(u),\quad c + \min(t_2(u), \lambda t_1(u)) v(u)).$$

Let us construct a continuous mapping $F \colon (A \setminus B) \times [0, 1] \to A \setminus B$. Any point $x \in A \setminus B$ lies on a unique ray corresponding to some parameter $u_x \in \operatorname{int} D$. It belongs to the interval $I_{u_x}$; let $c + m(u_x) v(u_x)$ be the midpoint of this interval. The point $x$ corresponds to a parameter $\tau_x$ on the ray. We define a homotopy that contracts the interval to its midpoint:
$$F(x, s) = c + ((1 - s)\tau_x + s m(u_x)) v(u_x).$$

At $s = 1$, the entire set $A \setminus B$ is contracted to the set $M = \{c + m(u) v(u) : u \in \operatorname{int} D\}$. The set $M$ is homeomorphic to the open disk $\operatorname{int} D$ in $\mathbb{R}^{n-1}$ (since $u \mapsto c + m(u) v(u)$ defines a homeomorphism). Because an open disk is contractible to a point, $M$, and therefore $A \setminus B$, is also contractible to a point.

\qed 
\vspace{1em} 

\begin{defin}
We say that a normed space satisfies the \textbf{convex cone property} if for any closed ball $B=\{x\;:\;\|x-x_B\|\leq r_B\}$ and open ball $A=\{x\;:\;\|x-x_A\|< r_A\}$, with $r_A<r_B$, such that the difference $A \setminus B$ is non-empty, the cone of directions from the center of $B$ through the set $A \setminus B$ is convex.
\end{defin}

\begin{lem}
Any normed plane satisfies the convex cone property.
\end{lem}

\noindent\textit{Proof.}

Let us introduce a polar coordinate system centered at the point $x_B$. We choose the coordinate axes such that the polar angle $\varphi = 0$ corresponds to the direction from $x_B$ to the center of the small ball $x_A$. Any ray emanating from $x_B$ is uniquely determined by a direction vector $v(\varphi) = (\cos\varphi, \sin\varphi)$, where $\varphi \in (-\pi, \pi]$ is the polar angle.

Let $\Phi$ be the set of all angles $\varphi$ for which the corresponding ray $R_\varphi$ intersects the set $A \setminus B$. Since the center $x_A$ lies at the angle $\varphi = 0$ and the ball $A$ does not cover the origin, the directions intersecting $A$ (and thus $A \setminus B$) do not cross the ray $\varphi = \pi$. We will show that $\Phi$ is an interval of length strictly less than $\pi$.

First, by the previous lemma, the set $A \setminus B$ is contractible, and therefore path-connected. If we take any two points in $A \setminus B$ and connect them by a continuous curve $\gamma$, then the projection of this curve onto the unit circle of directions from the center $x_B$ will also be continuous. Consequently, the set of directions $\Phi$ is path-connected on the circle, meaning it forms an arc.

Second, let us show that this arc does not contain opposite directions (i.e., pairs of angles $\varphi$ and $\varphi + \pi$). Assume the contrary: there exist two rays from $x_B$ in opposite directions $v$ and $-v$, both of which intersect $A \setminus B$.
This means that on these rays there are points $y_1 = x_B + t_1 v$ and $y_2 = x_B - t_2 v$ belonging to $A \setminus B$.

Since $y_1, y_2 \notin B$, their distances to the center of $B$ are strictly greater than the radius $r_B$: $t_1 > r_B$ and $t_2 > r_B$.
The distance between these points is:
$|y_1 - y_2| = |(x_B + t_1 v) - (x_B - t_2 v)| = |(t_1 + t_2)v| = t_1 + t_2 > 2r_B$.

On the other hand, both points $y_1, y_2$ lie in the ball $A$ of radius $r_A$. By the triangle inequality, the distance between any two points of the ball $A$ does not exceed its double radius (diameter):
$|y_1 - y_2| \le |y_1 - x_A| + |x_A - y_2| < r_A + r_A = 2r_A$.

We obtain a contradiction: $2r_B < |y_1 - y_2| < 2r_A$, which implies $r_B < r_A$, contradicting the condition of the lemma ($r_A < r_B$).

Thus, the set of angles $\Phi$ is a connected arc containing no opposite directions. Hence, its length is strictly less than $\pi$, and it lies entirely in some interval of the form $(\varphi_0 - \alpha, \varphi_0 + \alpha)$, where $\alpha < \pi/2$. The cone $C$ formed by the rays with directions from such an interval is strictly convex by definition.

\qed 
\vspace{1em} 

\begin{lem}
In a finite-dimensional normed space, let there be a closed convex bounded set $A$ with a non-empty interior, and a convex cone $C=\{tv : v\in V\}$ with the apex at the origin, where the set of directions $V$ is a subset of the unit sphere $S^{n-1}$. Assume that the intersection $C\cap \operatorname{int} A$ is non-empty. Then the functions $t_{in}(v) = \inf\{t > 0 : tv\in \operatorname{int} A\}$ and $t_{out}(v) = \sup\{t > 0 : tv\in \operatorname{int} A\}$, which map a direction $v \in V$ to the entry and exit times of the cone's rays on $A$, are continuous with respect to the argument $v$ on the set $V$.
\end{lem}

\noindent\textit{Proof.}
Let us denote the functions of the entry and exit points for a direction $v \in V$ as:

$$t_{in}(v)=\inf\{t>0\;|\;tv\in A\},\quad t_{out}(v)=\sup\{t>0\;|\;tv\in A\}.$$

Since the set $A$ is closed, bounded, and convex, the intersection of any ray $tv$ with the set $A$ is a segment $[t_{in}(v), t_{out}(v)]$. By the condition of the lemma, the considered rays intersect $\operatorname{int} A$, hence $t_{in}(v) < t_{out}(v)$, and for any $t \in (t_{in}(v), t_{out}(v))$ the point $tv$ lies strictly inside $A$ (in $\operatorname{int} A$).

We will prove the continuity of the function $t_{out}(v)$ at an arbitrary point $v_0 \in V$. To do this, we will show that for any $\varepsilon > 0$, a sufficiently small change in the direction $v$ results in a change of $t_{out}(v)$ by less than $\varepsilon$. Let us denote $t_0 = t_{out}(v_0)$.

Lower bound. Without loss of generality, we may assume that $\varepsilon > 0$ is small enough so that $t_0 - \varepsilon > t_{in}(v_0)$ (if the continuity holds for small $\varepsilon$, it holds for any $\varepsilon$). Consider the point $x_- = (t_0 - \varepsilon) v_0$. Since it lies on the ray strictly between the entry and exit points, $x_- \in \operatorname{int} A$. Because $\operatorname{int} A$ is an open set, there exists a neighborhood of the direction $v_0$ on the sphere of directions such that the point $(t_0 - \varepsilon)v$ still remains inside $\operatorname{int} A$. Consequently, for all $v$ from this neighborhood, the ray exits $A$ no earlier than this time: $t_{out}(v) \geqslant t_0 - \varepsilon$.

Upper bound. Consider the point $x_+ = (t_0 + \varepsilon) v_0$. By the definition of the exit point, this point no longer belongs to $A$. Since the set $A$ is closed, its complement $\mathbb{R}^n \setminus A$ is an open set. Consequently, there exists a (possibly different) neighborhood of the direction $v_0$ for which the point $(t_0 + \varepsilon)v$ still lies in the open complement to $A$. By the convexity of $A$, the intersection of the ray with it is a connected segment. Since the point $(t_0 + \varepsilon)v$ lies strictly outside $A$, the ray must leave the set earlier: $t_{out}(v) \leqslant t_0 + \varepsilon$.

The intersection of these two neighborhoods defines a neighborhood of $v_0$ in which the following holds:

$$t_0-\varepsilon\leq t_{out}(v)\leq t_0+\varepsilon.$$

This means that the function $t_{out}(v)$ is continuous at the point $v_0$.

The continuity of the entry function $t_{in}(v)$ is proved absolutely similarly: the point $(t_{in}(v_0) + \varepsilon) v_0$ lies in the open set $\operatorname{int} A$, which gives $t_{in}(v) \leqslant t_{in}(v_0) + \varepsilon$, while the point $(t_{in}(v_0) - \varepsilon) v_0$ lies in the open complement $\mathbb{R}^n \setminus A$, which gives $t_{in}(v) \geqslant t_{in}(v_0) - \varepsilon$.
\qed 
\vspace{1em} 

\begin{theorem}
In a normed space that satisfies the convex cone property, let there be a closed ball $B=\{x\;:\;\|x-x_B\|\leq R\}$ and a set $A=\bigcap^n_{i=1}A_i$, where $A_i=\{x\;:\;\|x-x_{A_i}\|< r_i\}$ with $r_i<R$, which is an intersection of open balls. If the set $A \setminus B$ is non-empty, then it is contractible.
\end{theorem}

\noindent\textit{Proof.}

Let us place the origin at the center of the large ball $x_B = 0$. Then $B = \{x : \|x\| \le R\}$.
For any direction $v$, let $R_v(t) = tv$ ($t > 0$) denote the ray emanating from the origin. For each open ball $A_i$, the ray intersects it along the interval $(t_{in, i}(v), t_{out, i}(v))$.

By the condition of the theorem, the space satisfies the convex cone property. This means that for each ball $A_i$, the set of directions whose rays intersect $A_i \setminus B$ forms a strictly convex cone. Let us denote it by $C_i$. A ray belongs to $C_i$ if and only if it intersects $A_i$ and the furthest point of intersection of the ray with the closure of $A_i$ lies strictly outside $B$, that is, $t_{out, i}(v) > R$.

The set $A = \bigcap_{i=1}^n A_i$ is open and convex as an intersection of open balls. Let $C_A$ denote the cone of directions of the rays intersecting $A$. By the convexity of $A$, the cone $C_A$ is also convex. For any direction $v \in C_A$, the ray intersects $A$ along the interval $(t_{in}(v), t_{out}(v))$, where:

$$t_{in}(v)=\max_i t_{in,i}(v),\quad t_{out}(v)=\min_i t_{out,i}(v).$$

Let us determine which directions constitute the final cone $C$, corresponding to the rays intersecting $A \setminus B$. A ray intersects $A \setminus B$ if and only if it intersects $A$ ($v \in C_A$) and the furthest point of the closure of $A$ on this ray lies strictly outside $B$:

$$t_{out}(v)>R\Longleftrightarrow \min_i t_{out,i}(v)>R\Longleftrightarrow \forall i\;\;t_{out,i}(v)>R$$

\begin{wrapfigure}{r}{0.55\textwidth}
\centering
\vspace{-10pt}
\begin{tikzpicture}[scale=0.8]
    \coordinate (Origin) at (0,0); 
    \def\RB{3} 
    
    \coordinate (A1) at (2.8, 1.2);
    \def\Rone{2}
    
    \coordinate (A2) at (1.2, 2.8);
    \def\Rtwo{2}

    \begin{scope}
        \clip (A1) circle (\Rone);
        \clip (A2) circle (\Rtwo);
        \fill[red!20, even odd rule] (Origin) circle (\RB) (-5,-5) rectangle (6,6);
    \end{scope}

    \draw[thick, orange] (Origin) circle (\RB);
    \draw[thick, blue] (A1) circle (\Rone);
    \draw[thick, green!70!black] (A2) circle (\Rtwo);

    \draw[thick, dotted, gray] (Origin) -- (28.18:5.5);
    \draw[thick, dotted, gray] (Origin) -- (61.82:5.5);

    \fill (28.18:\RB) circle (2pt);
    \fill (61.82:\RB) circle (2pt);

    \draw[ultra thick, purple] plot [smooth] coordinates {
        (28.18:3)       
        (35:3.38)       
        (45:3.738)      
        (55:3.38)       
        (61.82:3)       
    };
    
    \node[purple, font=\Large\bfseries] at (52:4.2) {$S_{mid}$};

    \coordinate (T1) at (45:3);       
    \coordinate (T2) at (45:4.477);   
    \coordinate (M) at (45:3.738);    
    \coordinate (X) at (45:3.2);      
    
    \draw[thick, dashed] (Origin) -- (45:5.5);
    
    \draw[line width=3.5pt, red, opacity=0.4] (T1) -- (T2);

    \fill (Origin) circle (2.5pt) node[below left] {$x_B=0$};
    \fill (T1) circle (2pt) node[below right] {$t_1(v)v$};
    \fill (T2) circle (2pt) node[above right] {$t_2(v)v$};
    \fill[purple] (M) circle (2.5pt) node[above left, text=purple] {$m(v_x)\cdot v_x$};
    \fill[red] (X) circle (2pt) node[below right=1pt, text=red] {$x$};

    \draw[->, thick, shorten >=2pt, shorten <=2pt] (X) -- (M);

    \node[orange, font=\Large\bfseries] at (-1.8, 2) {$B$};
    \node[blue, font=\Large\bfseries] at (4.2, 0.55) {$A_1$};
    \node[green!70!black, font=\Large\bfseries] at (0.7, 4.3) {$A_2$};

\end{tikzpicture}
\caption{}
\label{fig:homothety2}
\end{wrapfigure}

Consequently, a ray belongs to the cone $C$ if and only if it belongs to $C_A$ and to all cones $C_i$. That is, $C = C_A \cap \left( \bigcap_{i=1}^n C_i \right)$.
Since the intersection of any family of strictly convex cones is a strictly convex cone, the final cone $C$ is strictly convex.

Since the set $A \setminus B$ is non-empty, the cone $C$ is non-empty. Let us choose an arbitrary ray from $C$ and direct the abscissa axis $Ox_1$ of our Cartesian coordinate system along it. The remaining axes are chosen orthogonally.

Since $C$ is a strictly convex cone, it lies entirely in a single half-space $x_1 > 0$. We project the directions of the cone $C$ onto the hyperplane $x_1 = 1$. The intersection of the cone $C$ with this hyperplane is an open, convex, bounded set, which we denote by $\operatorname{int} D \subset \mathbb{R}^{n-1}$.
The set $\operatorname{int} D$ is homeomorphic to an open disk. To each point $u \in \operatorname{int} D$, there uniquely corresponds a unit direction vector $v(u)$.

For any parameter $u \in \operatorname{int} D$, the ray $R_u(t)$ intersects $A \setminus B$ along the open interval:

$$I_u=(t_1(u)v(u), t_2(u)v(u)),$$

where $t_1(u) = \max(R, t_{in}(u))$ is the entry point into $A \setminus B$, and $t_2(u) = t_{out}(u)$ is the exit point.

By the previous lemma (applied to the convex set $A$), the functions $t_{in}(u)$ and $t_{out}(u)$ are continuous on $\operatorname{int} D$. Hence, the functions $t_1(u)$ and $t_2(u)$ are also continuous.

Let us define the continuous function of the interval's midpoint: 
$$m(u)=\frac{t_1(u)+t_2(u)}{2}.$$

Any point $x \in A \setminus B$ lies on a unique ray corresponding to the parameter $u_x \in \operatorname{int} D$, and is at a distance $\tau_x = \|x\|$ from the origin.
Let us construct a continuous homotopy that contracts each interval to its midpoint:

$$H(x,s)=((1-s)\tau_x+s\cdot m(u_x))v(u_x).$$

At $s = 1$, the entire set $A \setminus B$ contracts to the "mid-surface" $S_{mid} = \{m(u) v(u) : u \in \operatorname{int} D\}$.
The mapping $u \mapsto m(u) v(u)$ defines a homeomorphism between the open convex set $\operatorname{int} D$ (which is homeomorphic to a disk) and the surface $S_{mid}$. Since an open disk is contractible to a point, $S_{mid}$, and consequently the entire set $A \setminus B$, is also contractible to a point.

\qed
\vspace{1em}

\begin{theorem}
Let a finite-dimensional normed space $X$ be equipped with a cylindrical norm $\|(x,y,z_1,\ldots,z_k)\|=\max\{\|(x,y)\|_{X_2},|z_1|,\ldots,|z_k|\}$, where $\|\cdot\|_{X_2}$ is a norm in some normed plane. Then the statement of the previous theorem holds for this norm.
\end{theorem}

\noindent\textit{Proof.}

For the sake of simplicity, we will present the proof for a three-dimensional space; for higher dimensions, the proof is completely analogous. Balls in the cylindrical norm have the form $D \times I$, where $D$ is a ball in the $X_2$ norm, and $I$ is a segment along the $z$-axis. 

Let 
$$B=D_B\times I_B,\quad A=\bigcap A_i=\bigcap (D_i\times I_i)=\bigcap D_i\times \bigcap I_i=D_A\times I_A$$

Since $I_A$ is an open interval and $I_B$ is a closed segment, their difference $I_A \setminus I_B$ consists of at most two disjoint open intervals, say $I_{top}$ and $I_{bot}$ (one or both may be empty). Consequently, the difference $A \setminus B$ can be represented as a union of sets:
$$
A \setminus B = (D_A\times I_A)\setminus (D_B\times I_B)=
((D_A\setminus D_B)\times I_A)\cup (D_A\times I_{top}) \cup (D_A\times I_{bot}).
$$

Let us denote $U = (D_A\setminus D_B)\times I_A$, and $V_{top} = D_A\times I_{top}$, $V_{bot} = D_A\times I_{bot}$. 

1. The set $U$ is contractible because it is a Cartesian product of contractible sets: $D_A\setminus D_B$ (which is contractible by the theorem for the plane, assuming it is non-empty) and the interval $I_A$. 
2. The sets $V_{top}$ and $V_{bot}$ are contractible since $D_A$ is convex, and $I_{top}$, $I_{bot}$ are intervals. Note that $V_{top}$ and $V_{bot}$ are disjoint.
3. Consider the intersection $U \cap V_{top} = (D_A\setminus D_B)\times I_{top}$. If $V_{top}$ is non-empty, this intersection is also non-empty and contractible as a product of contractible sets. Therefore, by the theorem on the union of two contractible sets with a contractible intersection (\cite{Fomenko}), the union $U \cup V_{top}$ is contractible.
4. Now consider the intersection of $(U \cup V_{top})$ with $V_{bot}$. Since $V_{top}$ and $V_{bot}$ are disjoint, this intersection is exactly $U \cap V_{bot} = (D_A\setminus D_B)\times I_{bot}$. Again, if $V_{bot}$ is non-empty, this intersection is non-empty and contractible. Applying the same theorem, we conclude that $(U \cup V_{top}) \cup V_{bot} = A \setminus B$ is contractible.

\qed

\label{end}

\end{document}